\documentclass[11pt, reqno]{amsart}
\usepackage{amsmath, amssymb, amsthm}
\usepackage{amsfonts, amscd, bm, cancel}
\usepackage{amsrefs}

\usepackage[margin=1.3in]{geometry}
\usepackage{setspace}

\usepackage{float}

\usepackage{xcolor}

\usepackage[colorlinks=true, urlcolor=blue, linkcolor=blue, citecolor=magenta]{hyperref}

\newtheorem*{thmA}{Theorem A}
\newtheorem*{assumA}{Assumption A}

\newtheorem{theorem}{Theorem}[section]

\newtheorem{lemma}[theorem]{Lemma}
\newtheorem{proposition}[theorem]{Proposition}
\newtheorem{definition}{Definition}[section]
\newtheorem{corollary}[theorem]{Corollary}

\newtheorem{remark}[theorem]{Remark}
\theoremstyle{remark}

\newtheorem*{ack}{Acknowledgments}

\numberwithin{equation}{section}
\allowdisplaybreaks

\begin{document}

\title[Alexandrov-Fenchel deficit]{Upper bounds for the Alexandrov-Fenchel deficit via integral formulas}

\author[K.-K. Kwong]{Kwok-Kun Kwong}
\address{School of Mathematics and Applied Statistics,
University of Wollongong\\
NSW 2522, Australia}
\email{\href{mailto:kwongk@uow.edu.au}{kwongk@uow.edu.au}}

\author[Y. Wei]{Yong Wei}
\address{School of Mathematical Sciences, University of Science and Technology of China, Hefei 230026, P.R. China}
\email{\href{mailto:yongwei@ustc.edu.cn}{yongwei@ustc.edu.cn}}
\date{\today}

\begin{abstract}
We derive a number of sharp upper bounds for the deficit in the Alexandrov-Fenchel inequality using a weighted Minkowski integral formula and an integral formula for the deficit in Jensen's inequality. Our estimates yield results under weaker convexity assumptions compared to approaches based on inverse curvature flows. The use of weighted formulas provides flexibility in deriving inequalities with different weight functions. Furthermore, our estimates are more quantitative as they include a distance term measuring the domain's deviation from a reference ball. We also analyze the stability of a weighted geometric inequality from a recent paper \cite{kwong2023geometric} via analysis of the support function on the sphere and show that, with an optimal choice of the origin, this inequality is stronger than the classical isoperimetric inequality.
\end{abstract}

\maketitle

\tableofcontents

\section{Introduction}

Geometric inequalities play a fundamental role in the study of hypersurfaces in space forms. Classical examples include the Alexandrov-Fenchel inequalities, the isoperimetric inequality, Minkowski inequalities, Brunn-Minkowski inequality and Willmore inequality, which have applications in convex geometry, curvature flows, eigenvalue estimates, and geometric analysis. In this paper, we derive a series of sharp geometric inequalities for closed hypersurfaces in space forms, utilizing a weighted version of the Minkowski integral formulas \cite{kwong2018weighted} and an integral formula for the deficit in Jensen's inequality. By selecting different weight functions, this approach enables the derivation of a number of sharp geometric inequalities, highlighting its flexibility. In Euclidean space, our results yield a number of families of sharp upper bounds for the deficit in the Alexandrov-Fenchel inequalities, refining existing inequalities and providing new stability results.

To set up our framework, we begin by introducing some notation and fundamental geometric quantities. Let $\Omega\subset \mathbb{R}^{n+1}$ be a smooth bounded domain enclosed by an embedded hypersurface $\Sigma=\partial\Omega$. The quermassintegrals of $\Omega$ can be expressed (up to a constant) as follows (\cite[\S 5.3]{Schneider2014}):
\begin{equation*}
I_{-1}=(n+1)|\Omega|, \quad I_k=\int_\Sigma H_kd\mu, ~k=0, 1, \cdots, n,
\end{equation*}
where $H_k=\binom{n}{k}^{-1}\sigma_k(\kappa)$ denotes the (normalized) $k$-th mean curvature of $\Sigma$. The classical Alexandrov-Fenchel inequality states that
\begin{thmA}[see {\cite[\S 7]{Schneider2014}}]
For a smooth convex bounded domain $\Omega\subset \mathbb{R}^{n+1}$, the quermassintegrals satisfy
\begin{equation}\label{s1.AF1}
I_{k-1}^2 - I_k I_{k-2} \ge 0, \quad \forall~k=1, \cdots, n-1
\end{equation}
or equivalently,
\begin{equation}\label{s1.AF2}
\frac{I_{k-1}}{I_k} - \frac{I_{k-2}}{I_{k-1}} \ge 0, \quad \forall~k=1, \cdots, n-1
\end{equation}
with equality holds if and only if $\Omega$ is a round ball.
\end{thmA}
\noindent This result is a special case of the Alexandrov-Fenchel inequality for mixed volumes in the theory of convex bodies, cf. \cite[(7.63)]{Schneider2014}.

In this paper, motivated by \eqref{s1.AF2}, we first derive a formula for the difference $\varphi\left(\frac{I_{k-1}}{I_k}\right) - \varphi\left(\frac{I_{k-2}}{I_{k-1}}\right)$ for an arbitrary function $\varphi$ (see Proposition \ref{prop rn comp}). By considering different functions $\varphi$, this formula leads to a series of sharp geometric inequalities, generally of the form:
\begin{equation*}
\boxed{ \quad \text{deficit} + \text{distance}^2 \le \text{integral} \quad}
\end{equation*}
where the ``deficit'' refers to the deficit $I_{k-1}^2 - I_k I_{k-2}$ (or alternatively, the ratio $\frac{I_{k-1}^2}{I_k I_{k-2}}$) from the Alexandrov-Fenchel inequality. The ``distance'' represents a measure of how far $\Omega$ deviates from a particular reference ball, while ``integral'' is a non-negative integral $\mathrm{I}$ over $\partial \Omega$, provided that $\partial \Omega$ satisfies certain convexity conditions. A key property is that the non-negative integral $\mathrm{I}$ vanishes precisely when $\Omega$ is a ball centered at the origin. As a particular example, we prove
\begin{theorem}\label{s1.thm1}
Let $\Sigma=\partial \Omega$ be a smooth, closed hypersurface in $\mathbb{R}^{n+1}$ with  $H_k \ge 0$. Then
\begin{align*}
\frac{I_{k-1}^2- I_k I_{k-2}}{I_k}+\delta_{2, k}\left(\Omega, B_0\left(\frac{I_{k-1}}{I_k}\right)\right)^2=\frac{1}{k\binom {n}{k}} \int_{\Sigma} T_{k-1}\circ A\left(X^T, X^T\right)d\mu,
\end{align*}
where $\delta_{2, k}$ is the weighted $L^2$ distance (see Section \ref{sec. identity}), $T_{k}$, $A$ are the $k$-th Newton tensor and shape operator of $\Sigma$ respectively, and $X^T$ is the tangential component of the position vector $X$. In particular, we have an upper bound for the deficit in the Alexandrov-Fenchel inequality:
\begin{equation}\label{s1. ineq mink deficit}
\frac{I_{k-1}^2- I_k I_{k-2}}{I_k}\le\frac{1}{k\binom{n}{k}} \int_{\Sigma} T_{k-1} \circ A\left(X^T, X^T\right) d \mu.
\end{equation}
If $H_k>0$, then the equality holds if and only if $\Sigma$ is a sphere centered at the origin.
\end{theorem}

For additional results on sharp upper bounds for the deficit in Alexandrov-Fenchel inequalities, see Theorems \ref{thm mink deficit} through \ref{thm um rn}. While some similar inequalities have been previously obtained via inverse curvature flows, we employ an alternative approach based on weighted Minkowski integral formulas, which leads to improved quantitative estimates and exact equalities in some cases.

As an application of Theorem \ref{s1.thm1}, we provide a stability result of the following geometric inequality:
\begin{equation}\label{s1.kw}
\int_{\Sigma} H_n |X|^2 \, d\mu - \frac{n+1}{\omega_n}I_{n-1}^2 + n I_{n-2}~\ge ~0
\end{equation}
for a closed, convex hypersurface in $\mathbb{R}^{n+1}$, proved by the authors \cite[Theorem 1.3]{kwong2023geometric} using inverse curvature flow.

\begin{corollary}\label{s1.cor1}
Let $\Sigma=\partial \Omega$ be a smooth, closed convex hypersurface in $\mathbb{R}^{n+1}$. Then
\begin{equation*}
n\delta_2\left(\Omega, B_0\left(\frac{I_{n-1}}{\omega_n}\right)\right)^2 \le \int_\Sigma H_n|X|^2d\mu-\frac{n+1}{\omega_n} I_{n-1}^2+n I_{n-2},
\end{equation*}
where $\delta_2$ is the $L^2$ distance between two convex bodies (see \eqref{s3. defL2} for the definition). The equality holds if and only if $\Sigma$ is a sphere centered at the origin.
\end{corollary}

In Section \ref{sec. weight}, we will provide another proof of \eqref{s1.kw} by using a completely different method based on the analysis of the support function on the sphere, which additionally offers an explicit interpretation of the deficit in \eqref{s1.kw} in terms of the $L^2$ distance from the ball $B_0\left(\frac{I_{n-1}}{\omega_n}\right)$. The result states as follows:

\begin{theorem}\label{s1.thm2}
For a closed convex hypersurface $\Sigma=\partial \Omega$ in $\mathbb R^{n+1}$,
\begin{equation*}
\int_{\Sigma} H_{n} |X|^2 d \mu - \frac{n+1}{\omega_n}I_{n-1}^2 + n I_{n-2}
=(n+1)\delta_2\left(\Omega, B_0(r)\right)^2.
\end{equation*}
Here $B_0(r)$ is the ball centered at 0 with the radius $r=\frac{I_{n-1}}{\omega_n}$, which also equals to the half of the mean width $\overline{w}(\Omega)$ of $\Omega$.
\end{theorem}

Theorem \ref{s1.thm2} serves a dual purpose: it strengthens inequality \eqref{s1.kw}, offering a new proof of a result in \cite{kwong2023geometric}, while also quantifying its deficit as the $L^2$ distance between $\Omega$ and a ball with the same mean width.

\begin{remark}
The inequality \eqref{s1.kw} serves as an example illustrating three distinct approaches to proving geometric inequalities:
\begin{itemize}
\item the geometric flow method (\cite[Theorem 1.3]{kwong2023geometric}),
\item the integral method (Corollary \ref{s1.cor1}), and
\item analysis of the support function on the sphere (Theorem \ref{s1.thm2}).
\end{itemize}
Each of these methods offers unique advantages.

The geometric flow method, with inverse curvature flow as an example, is a well-established approach with extensive results on existence, uniqueness, limiting behavior, and the preservation of various convexity conditions. This makes it a natural tool for studying geometric inequalities, particularly those related to different types of curvatures.
However, deriving Alexandrov-Fenchel type inequalities like \eqref{s1.AF1} via flows remains challenging. Flows can often be designed to preserve one geometric quantity while making another monotone, but extending this approach to inequalities involving three or more quantities presents additional difficulties, as it requires maintaining control over all terms. Moreover, inverse curvature flow requires star-shapedness for long-time existence, which imposes restrictions on the topological type of $\Sigma$ in which such inequalities can be proved.

The integral method, based on weighted Minkowski integral formulas, offers several advantages:
\begin{itemize}
\item It does not require star-shapedness and often assumes weaker convexity.
\item It provides more quantitative estimates and, in some cases, leads to exact equations rather than just inequalities.
\item This enables a more precise stability analysis, making it particularly useful for studying the stability of geometric inequalities.
\end{itemize}
These properties make the integral approach effective in deriving inequalities involving multiple curvatures or more than two geometric terms.

The method based on analysis of the support function on the sphere, on the other hand, provides the most precise information, at least for this inequality. It yields an exact formula for the deficit, quantifying the deviation from the reference sphere using an $L^2$ distance. However, its applicability appears to be more limited, particularly in higher dimensions, where it seems to be most effective in the convex setting. This limits the geometric and topological types of $\Sigma$ to which this method can be applied.

\end{remark}

The deficit in \eqref{s1.kw} clearly depends on the choice of the origin; however, we can show that it is minimized when the origin is translated to the ``Steiner point'' of $\Omega$. This follows from a minimization property of the $L^2$ distance associated with the Steiner ball $B_z(\Omega)$ (see Corollary \ref{cor steiner pt}).

Furthermore, using the known result that the isoperimetric deficit is bounded below by the $L^2$ distance $\delta_2(\Omega, B_z(\Omega))$ (see \cite[Theorem 5.3.1]{Groemer1996}), we prove that, within the class of convex bodies, inequality \eqref{s1.kw} is stronger than the classical isoperimetric inequality after an optimal choice of the origin.

\begin{corollary}
Let $\Sigma=\partial \Omega$ be a closed smooth convex hypersurface in $\mathbb{R}^{n+1}$. Assume (by translation) that the Steiner point of $\Omega$ is $0$, then
\begin{align*}
& \int_{\Sigma} H_{n}|X|^2 d \mu - \frac{n+1}{\omega_n}I_{n-1}^2 + n I_{n-2}\\
\le&
\frac{(n+1)I_{n-2}(\Omega)}{\eta_n |\Omega|^{n}} \left[\left(\frac{|\Sigma|}{\left|\mathbb S^{n}\right|} \right)^{n+1} - \left(\frac{ |\Omega| }{ \left| \mathbb B^{n+1} \right| } \right)^n \right],
\end{align*}
where $\eta_n=\frac{n+2}{n|\mathbb{B}^{n+1}|^n}$.
\end{corollary}

To summarize, we illustrate the range of inequalities and equations obtained from our approach by considering the case where $\Sigma$ is a strictly convex curve in $\mathbb{R}^2$. A summary of these results is provided in Table \ref{tab:inequalities} in Section \ref{sec upper bounds}.

The remainder of this paper is organized as follows. In Section \ref{sec prelim}, we introduce the necessary preliminaries, including relevant geometric notions and integral formulas that will be used throughout the paper. In Section \ref{sec. identity}, we derive an integral identity involving the support function and the $k$-th mean curvature, which serves as a foundation for our subsequent results. Section \ref{sec upper bounds} establishes sharp upper bounds for the deficit in the Alexandrov-Fenchel inequality, using a weighted Minkowski integral formula. In Section \ref{sec. weight}, we analyze the stability of a weighted geometric inequality, presenting an alternative proof via analysis of the support function on the sphere. Finally, in Section \ref{sec others}, we discuss additional inequalities obtained from our approach in other space forms or warped product manifolds under some natural assumptions.

\begin{ack}
The first author was supported by the UOW Early-Mid Career Researcher Enabling Grant and the UOW Advancement and Equity Grant Scheme for Research 2024. The second author was supported by National Key Research and Development Program of China 2021YFA1001800 and 2020YFA0713100, and the Fundamental Research Funds for the Central Universities.

Part of this research was completed during the authors' visit to the Australian National University and their participation in the MATRIX-MFO Tandem Workshop: Nonlinear Geometric Diffusion Equations at the MATRIX Institute in Australia. They would like to express their gratitude to Ben Andrews, Theodora Bourni, and Mat Langford for their kind invitations and acknowledge the support provided by ANU and the MATRIX Institute. They also thank Hojoo Lee for valuable discussions.
\end{ack}

\section{Preliminaries}\label{sec prelim}
Let $\left(\bar{M}^{n+1}, \bar{g}\right)$ be a warped product manifold defined as $I \times N^n$ with the metric
$$ \bar{g}=d r^2+\lambda^2(r) g_N, $$
where $\left(N^n, g_N\right)$ is a closed, smooth orientable Riemannian manifold, and $\lambda:I\to \mathbb R$ is a smooth function which is positive on the interior of $I$.
As examples, the space forms can be viewed as warped product manifolds $\bar{M}^{n+1}=I \times \mathbb{S}^n$ for an interval $I$:
\begin{equation}\label{eq space form}
\bar{M}^{n+1}=
\begin{cases}
\mathbb{R}^{n+1}, & \text { if } I=[0, \infty), \quad \lambda(r)=r \\
\mathbb{H}^{n+1}, & \text { if } I=[0, \infty), \quad \lambda(r)=\sinh (r) \\
\mathbb{S}^{n+1}, & \text { if } I=[0, \pi], \quad \lambda(r)=\sin (r).
\end{cases}
\end{equation}
We denote the simply connected space forms of curvature $K$ by $\bar M^{n+1}(K)$.

Warped product manifolds possess a conformal Killing vector field $X$ associated with the warping function $\lambda(r)$. Define the vector field $X=\lambda(r) \partial_r$. It satisfies the relation (see \cite[Lemma 2.2]{Brendle2013})
\begin{equation}\label{s2.alpha}
\mathcal{L}_X \bar{g}=2 \alpha \bar{g}, \text{ where }\alpha=\lambda^{\prime}(r)
\end{equation}
is the conformal factor, depending only on $r$. This conformal structure plays a fundamental role in deriving integral identities for hypersurfaces.
Let $\Sigma \subset \bar{M}^{n+1}$ be a closed, oriented hypersurface with unit normal vector field $\nu$. The normalized $k$-th mean curvature $H_k$ of $\Sigma$ is defined as the normalized symmetric polynomial of the principal curvatures $\left\{\kappa_1, \ldots, \kappa_n\right\}$:
$$
H_k=\frac{1}{\binom{n}{k}} \sigma_k\left(\kappa_1, \ldots, \kappa_n\right)
$$
where $\sigma_k$ is the $k$-th elementary symmetric polynomial. By convention, we also define
\begin{align*}
H_{-1}=\langle X, \nu\rangle.
\end{align*}
This convention is consistent with the Minkowski formula, which will be presented below.

For $k=1$, $H_1$ is the mean curvature, and for $k=n$, $H_n$ is the Gauss-Kronecker curvature. In addition to the mean curvatures, the $k$-th Newton tensor $T_k$ is defined recursively as
$$
T_0=\mathrm{Id}, \quad T_k=H_k \mathrm{Id}-A \circ T_{k-1}
$$
where $A$ is the shape operator of $\Sigma$ with respect to $\nu$. Locally, $\left(T_k\right)_j^i$ is given by
$$
\left(T_k\right)_j^i=\frac{1}{k!} \sum_{\substack{1 \le i_1, \ldots, i_k \le n \\
1 \le j_1, \cdots, j_k \le n}} \delta_{j j_1 \ldots j_k}^{i i_1 \ldots i_k} A_{i_1}^{j_1} \cdots A_{i_k}^{j_k}.
$$

\subsection{Weighted Minkowski formula}
We can now state the weighted Minkowski formula \cite[Proposition 1]{kwong2018weighted} for warped product manifolds, as follows:
\begin{proposition}
Let $f$ be a $C^1$ function on a closed hypersurface $\Sigma$ in $(\bar M, \bar g)$.
Then we have
\begin{equation}\label{weighted mink}
\int_{\Sigma} f\left(\alpha H_{k-1}-H_k\langle X, \nu\rangle\right)d\mu +\frac{1}{k\binom{n}{k}} \int_{\Sigma} f \mathrm{div}_{\Sigma}\left(T_{k-1}\right)(\xi)d\mu=-\frac{1}{k\binom{n}{k}} \int_{\Sigma}\left\langle T_{k-1}(\xi), \nabla_{\Sigma} f\right\rangle d\mu,
\end{equation}
where $\xi=X^T$ is the tangential component of $X$ on $\Sigma$, and $\mathrm{div}_{\Sigma}$ denotes the divergence operator on $\Sigma$.
\end{proposition}
In particular, it is well-known that the Newton tensor $T_k$ is divergence-free in space forms or when $k=1$ and $\bar M$ is Einstein, due to the Codazzi equation. Consequently, the weighted Minkowski formula \eqref{weighted mink} simplifies to:
\begin{equation}\label{weighted mink space form}
\int_{\Sigma} f\left(\alpha H_{k-1}-H_k\langle X, \nu\rangle\right)d\mu =-\frac{1}{k\binom{n}{k}} \int_{\Sigma}\left\langle T_{k-1}(\xi), \nabla_{\Sigma} f\right\rangle d\mu
\end{equation}
if $\bar M$ is $\mathbb R^{n+1}$, $\mathbb H^{n+1}$, $\mathbb S^{n+1}$, or when $k=1$ and $\bar M$ is Einstein.

When $f=1$, it reduces to the classical Minkowski (or Hsiung-Minkowski) formula (\cite{Hs54,MR91}):
\begin{equation}\label{eq mink}
\int_{\Sigma} \alpha H_{k-1}d\mu= \int_{\Sigma} H_k\langle X, \nu\rangle d \mu.
\end{equation}

\subsection{Jensen deficit}
We will apply the weighted Minkowski formula \eqref{weighted mink} to relate the difference between two weighted curvature integrals, using a weight of the form $f=f(u)$, where $u=\langle X, \nu\rangle$ denotes the (generalized) support function. Subsequently, we will analyze the deficit in Jensen's inequality and then apply the Minkowski formulas to obtain a formula for $\varphi\left(\frac{I_{k-1}}{I_k}\right) - \varphi\left(\frac{I_{k-2}}{I_{k-1}}\right)$ in $\mathbb R^{n+1}$.

For this purpose, we recall the following basic formula for the deficit in Jensen's inequality, which is derived from the integral form of the remainder in Taylor's theorem.
\begin{lemma}\label{lem jensen}
Let $(\Omega, m)$ be a probability space, and let $u: \Omega \rightarrow \mathbb{R}$ be a bounded measurable function with $I=$ Range $(u)$, assumed to be an interval. Define $\bar{u}=\int_{\Omega} u d m$, and let $\psi: I \rightarrow \mathbb{R}$ be a $C^2$ function. Then
\begin{equation}\label{jensen deficit}
\int_{\Omega} \psi (u) \, d m-\psi(\bar{u})=\int_{\Omega} \int_{\bar{u}}^{u(x)} \psi^{\prime \prime}(t)(u(x)-t) d t d m(x).
\end{equation}
In particular,
\begin{align*}
\frac{1}{2} \inf_I \psi^{\prime \prime} \int_{\Omega}(u(x)-\bar{u})^2 d m \le & \int_{\Omega} \psi(u) d m-\psi\left(\int_{\Omega} u d m\right) \\
\le & \frac{1}{2} \sup_I \psi^{\prime \prime} \int_{\Omega}(u(x)-\bar{u})^2 d m.
\end{align*}
\end{lemma}
\begin{proof}
For reader's convenience, we include the proof here.
By Taylor's theorem with the integral form of the remainder,
$$
\psi(u(x))=\psi(\bar{u})+\psi^{\prime}(\bar{u})(u(x)-\bar{u})+\int_{\bar{u}}^{u(x)} \psi^{\prime \prime}(t)(u(x)-t) d t.
$$
Integrating both sides with respect to $m$, and noting that $u-\bar u$ has mean $0$, we can get the result.
\end{proof}

\section{An integral identity involving the support function and the $k$-th mean curvature}\label{sec. identity}

In this section, we derive a general integral identity involving support function and the $k$th mean curvature. For a smooth closed hypersurface $\Sigma=\partial \Omega$ in $\bar M^{n+1}$, besides the quermassintegral $I_k$ we define the following weighted curvature integral
\begin{align*}
J_k:=\int_\Sigma \alpha H_k d\mu,
\end{align*}
where the weight $\alpha$ is defined in \eqref{s2.alpha}. In particular, when $\bar M^{n+1}=\mathbb R^{n+1}$, then $J_k=I_k$ and $I_{-1}=(n+1)|\Omega|$. When $\bar M^{n+1}=\mathbb H^{n+1}$ or $\mathbb S^{n+1}$, inequalities involving the weighted curvature integrals $J_k$ are considered in \cite{Brendle2016, deLima2016, Ge2015, Girao2019, HL22, HLW22, SX19, Xia2014} etc.

To quantify the deviation of $\Omega$ from a reference ball centered at the origin, we define the following weighted $L^2$ distance.
\begin{definition}\label{s3. def1}
Let $\Omega$ be a domain in $\bar M^{n+1}$ with a smooth boundary $\Sigma = \partial \Omega$, where the $k$-th mean curvature $H_k \ge 0$. Let $B_0(r)$ denote the ball of radius $r$ centered at the origin in $\bar M^{n+1}$. We define the weighted $L^2$ distance $\delta_{2, k}$ between $\Omega$ and $B_0(r)$ as
$$\delta_{2, k} (\Omega, B_0(r)):= \left(\int_\Sigma \left(u - r\right)^2 H_k \, d\mu\right)^{\frac{1}{2}}, $$
where $u=\langle X, \nu\rangle$ denotes the (generalized) support function of $\Sigma$.

If $\bar M=\mathbb R^{n+1}$, the average of $u$ with respect to the measure of $\Sigma$ is given by $\bar{u} = \frac{\int_\Sigma u \, d\mu}{\int_\Sigma d\mu} = \frac{(n+1)|\Omega|}{|\Sigma|}$. Thus, when $k=0$, the variance of $u$ is equal to the squared distance $\delta_{2,0}^2$ between $\Omega$ and the ball centered at $0$ with radius $\frac{(n+1)|\Omega|}{|\Sigma|}$, up to a constant:
$$\mathrm{Var}(u) = \frac{1}{|\Sigma|} \int_\Sigma \left(u - \bar{u}\right)^2 \, d\mu = \frac{1}{|\Sigma|} \delta_{2, 0} \left(\Omega, B_0\left(\frac{(n+1)|\Omega|}{|\Sigma|}\right)\right)^2. $$

Let us also recall the $L^2$ distance between two convex bodies in $\mathbb R^{n+1}$. If $K_1$ and $K_2$ are convex bodies in $\mathbb{R}^{n+1}$ with support functions $u_{K_1}$ and $u_{K_2}$, respectively, regarded as functions defined on $\mathbb{S}^n$, then their $L^2$ distance is defined as follows (see \cite[(7.120)]{Schneider2014}):
\begin{equation}\label{s3. defL2}
\delta_2(K_1, K_2) = \left(\int_{\mathbb{S}^n} \left(u_{K_1}(\xi) - u_{K_2}(\xi) \right)^2 \, d\theta \right)^{\frac{1}{2}} = \left\| u_{K_1} - u_{K_2} \right\|_{L^2 (\mathbb{S}^n)}.
\end{equation}

For a convex hypersurface in $\mathbb R^{n+1}$, since the spherical measure pushed forward by the inverse Gauss map satisfies $d\theta = H_n \, d\mu$, the weighted $L^2$ distance $\delta_{2, n}$ is the $L^2$ distance between $\Omega$ and a ball:
$$\delta_{2, n} (\Omega, B_0(r)) = \delta_2(\Omega, B_0(r)). $$
Recall also that if $K, L$ are two convex bodies in $\mathbb R^{n+1}$, the Hausdorff distance $\delta(K, L)$ between $K$ and $L$ is defined by (\cite[(1.60)]{Schneider2014})
$$
\delta(K, L)=\inf \left\{r: K \subset L_{(r)}, L \subset K_{(r)}\right\}.
$$
Here, $K_{(r)}=K+r \mathbb{B}^{n+1}$ is the parallel body of $K$ at distance $r$. It is easily seen that $\delta(K, L)$ can also be defined in terms of the respective support functions of the two bodies by (\cite[Lemma 1.8.14]{Schneider2014})
$$
\delta(K, L)=\sup \left\{\left|u_K(\xi)-u_L(\xi)\right|: \xi \in \mathbb{S}^n\right\}=\left\|u_K-u_L\right\|_{\infty}.
$$
\end{definition}

In the following result, we prove a relationship between a curvature integral weighted by an arbitrary function $\varphi$ of the support function $u$ and the curvature integrals $I_k$ and $J_k$. By further imposing convexity or concavity conditions on $\varphi$, we can obtain more elegant inequalities \eqref{ineq supp1}, \eqref{ineq supp2}. Integral formulas involving $H_k$ and specific functions of $u$ have been considered by, for example, \cite{Alencar1998}, where such formulas are used to prove Minkowski-type theorems.
\begin{theorem}\label{thm weighted supp}
Let $\Sigma$ be a smooth, closed hypersurface in the space form $\bar M^{n+1}(K)$ with $H_k \ge 0$, $\int_{\Sigma} H_k \, d\mu > 0$, and let $\varphi$ be a $C^2$ function on $\mathrm{Range}(u)$. Then the following holds:
\begin{equation}\label{eq identity}
\int_{\Sigma} H_k \varphi(u) \, d\mu =I_k \varphi\left(\frac{J_{k-1}}{I_k}\right) +\int_{\Sigma}H_k(x)\int_{\frac{J_{k-1}}{I_k}}^{u(x)} \varphi^{\prime \prime}(t)(u(x)-t) d t \, d \mu(x).
\end{equation}
In particular:
\begin{enumerate}
\item
We have
\begin{align}\label{ineq hk min max}
I_k \varphi\left(\frac{J_{k-1}}{I_k}\right)+\frac{1}{2}\left(\min_I \varphi^{\prime \prime}\right) \delta_{2, k}\left(\Omega, B_0(\bar{u})\right)^2
\le \int_{\Sigma} H_k \varphi(u) d \mu \nonumber\\
\le I_k \varphi\left(\frac{J_{k-1}}{I_k}\right)+\frac{1}{2}\left(\max_I \varphi^{\prime \prime}\right) \delta_{2, k}\left(\Omega, B_0(\bar{u})\right)^2,
\end{align}
where $I=\mathrm{Range}(u)$ and $\bar u= \frac{J_{k-1}}{I_k}$.
\item
If $\varphi$ is convex, then
\begin{equation}\label{ineq supp1}
\int_{\Sigma} H_k \varphi(u) \, d\mu \ge I_k \varphi\left(\frac{J_{k-1}}{I_k}\right).
\end{equation}
If $H_k > 0$ and $\varphi$ is strictly convex, then equality holds if and only if $\Sigma$ is a geodesic sphere centered at $0$.
\item
If $\varphi$ is concave, then
\begin{equation}\label{ineq supp2}
\int_{\Sigma} H_k \varphi(u) \, d\mu \le I_k \varphi\left(\frac{J_{k-1}}{I_k}\right).
\end{equation}
If $H_k > 0$ and $\varphi$ is strictly concave, then equality holds if and only if $\Sigma$ is a geodesic sphere centered at $0$.
\end{enumerate}
\end{theorem}

\begin{proof}
By the Minkowski formula \eqref{eq mink} and Lemma \ref{lem jensen} applied to the probability measure $\frac{H_k d\mu}{I_k}$, we have
\begin{align*}
\frac{\int_{\Sigma} H_k \varphi(u) d \mu}{I_k} =& \varphi\left(\frac{\int_{\Sigma} H_k u d \mu}{I_k}\right) +
\frac{1}{I_k}\int_{\Sigma}H_k(x) \int_{\frac{\int_{\Sigma} H_k u d \mu}{I_k}}^{u(x)} \varphi^{\prime \prime}(t)(u(x)-t) d t d \mu(x)
\\
=& \varphi\left(\frac{J_{k-1}}{I_k}\right)+\frac{1}{I_k} \int_{\Sigma} H_k(x)\int_{\frac{J_{k-1}}{I_k}}^{u(x)} \varphi^{\prime \prime}(t)(u(x)-t) d t d \mu(x).
\end{align*}
From this \eqref{eq identity} follows.

It also follows that for $I=\mathrm{Range}(u)$ and $\bar u= \frac{J_{k-1}}{I_k}$ is the average of $u$ with respect to $\frac{H_k d \mu}{I_k}$,
\begin{align*}
I_k\varphi\left(\frac{J_{k-1}}{I_k}\right)+\frac{1}{2} \min_I \varphi^{\prime \prime} \int_{\Sigma}H_k (u-\bar{u})^2 d \mu
\le \int_{\Sigma} H_k \varphi(u) d \mu\\
\le I_k\varphi\left(\frac{J_{k-1}}{I_k}\right)+\frac{1}{2} \max_I \varphi^{\prime \prime} \int_{\Sigma}H_k (u-\bar{u})^2 d \mu,
\end{align*}
which is equivalent to \eqref{ineq hk min max}.

In particular, if $\varphi$ is convex, then
\begin{equation}\label{ineq convex}
\int_{\Sigma} H_k \varphi(u) d \mu \ge I_k \varphi\left(\frac{J_{k-1}}{I_k}\right).
\end{equation}
Furthermore, if $H_k > 0$ and $\varphi$ is strictly convex and the equality in \eqref{ineq convex} holds, then $u$ must be constant by the equality case of Jensen's inequality. It follows that $\Sigma$ is a sphere centered at $0$. The case where $\varphi$ is concave is similar.
\end{proof}

As an immediate consequence of Theorem \ref{thm weighted supp}, we obtain a sharp lower bound on the circumradius of a $k$-convex domain, without requiring the star-shapedness condition. This result can be compared with \cite[Theorem 1.3]{WU24}, where an additional assumption of star-shapedness is required.
\begin{corollary}
Let $\Sigma=\partial \Omega$ for a smooth domain $\Omega \subset \bar{M}^{n+1}(K)$ (assume also $\Omega \subset \mathbb{S}^{n+1}_+$, the hemisphere, if $K=1$), with the origin chosen to be the circumcenter of $\Omega$. Let $\rho$ be the circumradius of $\Omega$. If $H_k > 0$ on $\Sigma$, then
\begin{equation}\label{s3.cir}
    \lambda(\rho) \ge \frac{J_{k-1}}{I_k},
\end{equation}
where $\lambda$ is the function defined in \eqref{eq space form}. Equality holds in \eqref{s3.cir} if and only if $\Sigma$ is a geodesic sphere centered at $0$.
\end{corollary}

\begin{proof}
As $0$ is the circumcenter of $\Omega$, we have $\max_{\Sigma} |u| = \lambda(\rho)$.
By applying \eqref{ineq supp1} with $\varphi(u)=u^{2m}$ for $m\ge 1$,
\begin{equation}\label{s3.cir-pf1}
    \int_{\Sigma} H_k u^{2m} d\mu \ge \frac{{J_{k-1}}^{2m}}{{I_k}^{2m-1}}.
\end{equation}
Taking the $\tfrac{1}{2m}$-th power of both sides and letting $m\to\infty$ gives the inequality \eqref{s3.cir}.

If equality holds in \eqref{s3.cir}, then
\begin{align*}
    \left(\frac{J_{k-1}}{I_k}\right)^{2m}I_k=\lambda(\rho)^{2m}I_k
    \geq \int_{\Sigma} H_k u^{2m} d\mu \ge \frac{{J_{k-1}}^{2m}}{{I_k}^{2m-1}}.
\end{align*}
This implies that equality holds in \eqref{s3.cir-pf1}, and thus $\Sigma$ is a geodesic sphere centered at $0$.
\end{proof}

\begin{remark}
When $\varphi (x) \equiv 1$, \eqref{eq identity} is trivial. When $\varphi (x) = x$, \eqref{eq identity} reduces to the Minkowski formula. For $\varphi (x) = x^2$ and $\bar M=\mathbb R^{n+1}$, then  \eqref{eq identity} becomes
\begin{align}\label{Hk u^2}
\int_\Sigma H_k u^2d\mu= & \frac{\left(\int_\Sigma H_{k-1} d\mu\right)^2}{\int_\Sigma H_kd\mu}+\delta_{2, k-1}\left(\Omega, B_0\left(\frac{I_{k-1}}{I_k}\right)\right)^2\nonumber\\
\ge & \frac{\left(\int_{\Sigma} H_{k-1} d \mu\right)^2}{\int_{\Sigma} H_k d \mu}.
\end{align}
Especially, when $k=1$, this reduces to
\begin{equation}\label{k=1}
\begin{split}
\int_\Sigma H_1u^2d\mu= \frac{|\Sigma|^2}{\int_\Sigma H_1d\mu}+\delta_{2, 0}\left(\Omega, B_0\left(\frac{(n+1)|\Omega|}{|\Sigma|}\right)\right)^2
\ge \frac{|\Sigma|^2}{\int_\Sigma H_1d\mu}.
\end{split}
\end{equation}
This result can be compared to the inequalities proved in \cite[Theorem 1.6]{girao2020weighted} and \cite[Theorem 2]{kwong2014new} (in the star-shaped mean-convex case), which assert, respectively:

\begin{equation}\label{girao}
\int_{\Sigma} H_1 r^2 d\mu \ge \omega_{n} \Bigl(\frac{|\Sigma|}{\omega_{n}}\Bigr)^{\tfrac{n+1}{n}}
\end{equation}
and
\begin{equation}\label{km}
\int_{\Sigma} H_1 r^2 d\mu \ge (n+1) |\Omega|.
\end{equation}

On the one hand, the left-hand side of \eqref{k=1} is not larger than $\int_{\Sigma} H_1 r^2 \, d\mu$. On the other hand, in the star-shaped mean-convex case, the Minkowski's inequality \cite{GL09}
$$\int_{\Sigma} H_1 d \mu \ge \omega_{n}^{\frac{1}{n}}|\Sigma|^{\frac{n-1}{n}}$$ implies that the right-hand side of \eqref{k=1} is also not larger than that of \eqref{girao}. However, it should be noted that our result does not require star-shapedness. Additionally, the RHS of \eqref{k=1} is not smaller than that of \eqref{km} in the convex case by another Alexandrov-Fenchel inequality:
\begin{equation}\label{ineq af}
|\Sigma|^2 \ge \left(n+1\right) \left(\int_{\Sigma} H_1 \, d\mu\right) |\Omega|.
\end{equation}
Therefore, we expect that the inequality $\int_{\Sigma} H_1 u^2 d \mu\ge \frac{|\Sigma|^2}{\int_{\Sigma} H_1 d \mu}$ in \eqref{k=1} is better than \eqref{km} in the mean convex case. However, to the best of our knowledge, the inequality \eqref{ineq af} remains unresolved in the mean-convex setting.

It is also interesting to compare \eqref{Hk u^2} to the following inequality by Kwong and Miao \cite{KM15}, who proved that for $k \in \{2, \ldots, n\}$, if $\Sigma \subset \mathbb{R}^{n+1}$ satisfies $H_k > 0$ on $\Sigma$,
\begin{equation}\label{km2}
\int_{\Sigma}H_k r^2 d\mu \ge \int_{\Sigma} H_{k-2} d\mu.
\end{equation}
This raises a natural question: does the inequality
$$\left(\int_\Sigma H_{k-1} d\mu\right)^2 \ge \int_\Sigma H_{k-2} d\mu \int_\Sigma H_k d\mu$$
hold for $k$-convex hypersurfaces? If this is true, then the inequality \eqref{Hk u^2} is stronger than \eqref{km2}, as the left-hand side of \eqref{Hk u^2} is smaller, while the right-hand side is larger. In the convex case, this inequality indeed holds, by the Alexandrov–Fenchel inequality for convex hypersurfaces (see Theorem A)
\end{remark}

\section{Sharp upper bounds for the Alexandrov-Fenchel deficit}\label{sec upper bounds}
In this section, inspired by the classical Alexandrov-Fenchel inequality \eqref{s1.AF1} in $\mathbb{R}^{n+1}$, we compare $\frac{I_{k-1}}{I_k}$ with $\frac{I_{k-2}}{I_{k-1}}$, or more generally, examine the difference $\varphi\left(\frac{I_{k-1}}{I_k}\right) - \varphi\left(\frac{I_{k-2}}{I_{k-1}}\right)$ for an arbitrary function $\varphi$. While the resulting formula may appear complex, we will demonstrate its utility by applying it to specific choices of $\varphi$, resulting in a number of simpler and more elegant inequalities.

\begin{proposition}\label{prop rn comp}
Let $\Sigma$ be a smooth, closed hypersurface in $\mathbb R^{n+1}$
with $H_k \ge 0$ and $H_{k-1} \ge 0$.
Let $\varphi$ be a $C^2$ function on $\mathrm{Range}(u)$ and define $\psi(t)=t\varphi(t)$. Then we have
\begin{align*}
& I_{k-1} \left(\varphi\left(\frac{I_{k-1}}{I_k}\right)- \varphi\left(\frac{I_{k-2}}{I_{k-1}}\right)\right)\\
& +\int_\Sigma \left(\int_{\frac{I_{k-1}}{I_k}}^{u(x)} H_k(x) \psi^{\prime \prime}(t)- \int_{\frac{I_{k-2}}{I_{k-1}}}^{u(x)} H_{k-1}(x) \varphi^{\prime \prime}(t) \right) (u(x)-t) d t d \mu\\
=& \frac{1}{k\binom{n}{k}} \int_{\Sigma} \varphi^{\prime}(u) T_{k-1} \circ A\left(X^T, X^T\right)d\mu.
\end{align*}
\end{proposition}
\begin{proof}
Note that in $\mathbb{R}^{n+1}$ we have $\alpha=1$ and $J_k=I_k$.  We begin by applying the weighted Minkowski formula \eqref{weighted mink space form}, which gives
\begin{align}\label{eq 1}
\int_{\Sigma} H_{k-1}\varphi(u) d\mu +
\frac{1}{k\binom{n}{k}}\int_{\Sigma} \varphi^{\prime}(u) T_{k-1} \circ A\left(X^T, X^T\right)\nonumber\\
=\int_{\Sigma} H_k\varphi(u) u d\mu = \int_{\Sigma} \psi(u) H_k d\mu,
\end{align}
where we have used $\nabla_{\Sigma}\left(\varphi(u) \right)= \varphi^{\prime}(u) A\left(X^T\right)$.

Next, we apply Theorem \ref{thm weighted supp} to both sides of \eqref{eq 1}. For the right-hand side, we have
\begin{align}\label{R.H.S. }
\mathrm{RHS}
=&I_k\psi\left( \frac{I_{k-1}}{I_k}\right)+ \int_{\Sigma} \int_{\frac{I_{k-1}}{I_k}}^{u(x)} H_k(x) \psi^{\prime \prime}(t)(u(x)-t) d t d\mu\nonumber\\
= &I_{k-1} \varphi\left(\frac{I_{k-1}}{I_k}\right) + \int_{\Sigma} \int_{\frac{I_{k-1}}{I_k}}^{u(x)} H_k(x) \psi^{\prime \prime}(t)(u(x)-t) d t d\mu.
\end{align}
Note also that the existence of an elliptic point on $\Sigma$ ensures that $I_k$ and $I_{k-1}$ are positive. Indeed, the point $p$ in $\Sigma$ that is furthest from the origin has all principal curvatures positive with respect to the outward normal.

On the other hand, again by applying Theorem \ref{thm weighted supp}, the first term on the LHS of \eqref{eq 1} becomes
\begin{equation}\label{L.H.S. first term}
\begin{aligned}
\int_{\Sigma} H_{k-1}\varphi(u) d\mu
& = I_{k-1} \varphi\left(\frac{I_{k-2}}{I_{k-1}}\right) + \int_{\Sigma} \int_{\frac{I_{k-2}}{I_{k-1}}}^{u(x)} H_{k-1}(x) \varphi^{\prime \prime}(t)(u(x)-t) d t d\mu.
\end{aligned}
\end{equation}
Substituting these expressions into \eqref{eq 1} completes the proof.
\end{proof}
We will apply Proposition \ref{prop rn comp} by selecting different functions for $\varphi$. The results we obtain (Theorem \ref{thm mink deficit} to Theorem \ref{thm um rn}), in a conceptual sense, generally take the form:
\begin{equation*}
\boxed{ \quad \text{deficit} + \text{distance}^2 \le \text{integral} \quad}
\end{equation*}
where the ``deficit'' refers to the deficit $I_{k-1}^2 - I_k I_{k-2}$ (or the ratio $\frac{I_{k-1}^2}{I_k I_{k-2}}$) from the Alexandrov-Fenchel inequality. The ``distance'' represents a measure of how far $\Omega$ deviates from a particular reference ball, while ``integral'' is a non-negative integral $\mathrm{I}$ over $\partial \Omega$, provided that $\partial \Omega$ satisfies certain convexity conditions. A key property is that the non-negative integral $\mathrm{I}$ vanishes precisely when $\Omega$ is a ball centered at the origin.

To derive such estimates, we will select $\varphi$ as an increasing function that is concave, but with $x \varphi(x)$ being convex (or vice versa if $\varphi$ is decreasing). Specifically, we will consider the following functions for $\varphi(x)$:
\begin{enumerate}
\item\label{f1} $x$
\item\label{f2} $\frac{1}{x}$
\item\label{f3} $x^m$ for $0 \le m < 1$
\item\label{f4} $x^{-m}$ for $0 \le m \le 1$
\item $\log x$
\item Solution to $\varphi^{\prime}(x)=\frac{1}{1+x^m}$ for $0<m\le 2$.
\end{enumerate}
Although functions \eqref{f1} and \eqref{f2} can be regarded as members of the families in \eqref{f3} and \eqref{f4} respectively, we separate them here as the assumptions for the corresponding results differ.

\begin{remark}
    To better illustrate the types of inequalities or equations that can be obtained, we consider the case where $\Sigma=\partial \Omega$ is a strictly convex curve in $\mathbb{R}^2$, with the circumcenter at the origin and $r$ denoting the circumradius. The derived inequalities and equations are summarized in the following Table \ref{tab:inequalities}. Here, $A$ and $L$ are the area and perimeter of $\Omega$ respectively; $\delta_{2, 1}=\delta_{2, 1}\left(\Omega, B_0\left(\frac{L}{2 \pi}\right)\right)$ and $\delta_{2, 0}=\delta_{2, 0}\left(\Omega, B_0\left(\frac{2 A}{L}\right)\right)$ represent the weighted $L^2$ distances between $\Omega$ and the respective reference balls.
\end{remark}

\begin{table}[H]
\centering
\small
\caption{Summary of inequalities and equations for convex curves}
\label{tab:inequalities}
\renewcommand\arraystretch{4}
\begin{tabular}{|c|c|}
\hline
\textbf{Inequality or Equation} & \textbf{Theorem} \\
\hline
$\displaystyle \frac{L^2-4 \pi A}{2 \pi}+{\delta_{2, 1}}^2 = \int_{\Sigma} \kappa\left|X^T\right|^2 ds$ & \ref{thm mink deficit}\\
\hline
$\displaystyle \frac{L^2-2 \pi A}{\pi}+{\delta_{2, 1}}^2 \le \int_{\Sigma} \kappa|X|^2 ds$ & \ref{thm Hn x2}\\
\hline
$\displaystyle \frac{L^2-4 \pi A}{2 A}+\frac{1}{r^3} {\delta_{2, 0}}^2 \le \int_{\Sigma} \frac{\kappa\left|X^T\right|^2}{u^2} ds$ & \ref{thm rn 1/u} (i)\\
\hline
$\displaystyle L \log \left (\frac{L^2}{4 \pi A}\right)+\frac{1}{2 r} {\delta_{2, 1}}^2+\frac{1}{2 r^2} {\delta_{2, 0}}^2 \le \int_{\Sigma} \frac{\kappa\left|X^T\right|^2}{u} ds$ & \ref{thm rn 1/u} (ii)\\
\hline
$\begin{aligned}\displaystyle (2 \pi)^{m-1} \cdot \frac{L^2-4 \pi A}{(2 \pi)^m+L^m}+\frac{2+(2-m) r^m}{2\left(1+r^m\right)^2} {\delta_{2, 1}}^2\\+\frac{m}{2\left(1+r^m\right)^2 r^{1-m}} {\delta_{2, 0}}^2 \le \int_{\Sigma} \frac{\kappa\left|X^T\right|^2}{1+u^m} d s\end{aligned}$ & \ref{thm tan rn} (i) ($0 < m \le 1$) \\
\hline
$\displaystyle (2 \pi)^{m-1} \cdot \frac{L^2-4 \pi A}{(2 \pi)^m+L^m}+\frac{2+(2-m) r^m}{2\left(1+r^m\right)^2} {\delta_{2, 1}}^2\le \int_{\Sigma} \frac{\kappa\left|X^T\right|^2}{1+u^m} d s$ & \ref{thm tan rn} (ii) ($1 < m \le 2$) \\
\hline
$\displaystyle \left (\frac{2 \pi}{L}\right)^{1-m} \frac{L^2-4 \pi A}{2 \pi}+\frac{m+1}{2 r^{1-m}} {\delta_{2, 1}}^2 +\frac{1-m}{2 r^{2-m}} {\delta_{2, 0}}^2 \le \int_{\Sigma} \frac{\kappa\left|X^T\right|^2}{u^{1-m}} d s$ & \ref{thm um rn} (i) ($0 \le m<1$)\\
\hline
$\displaystyle \left (\frac{2 A}{L}\right)^{1-m} \frac{L^2-4 \pi A}{2 A}+\frac{1-m}{2 r^{m+1}} {\delta_{2, 1}}^2+\frac{m+1}{2 r^{m+2}} {\delta_{2, 0}}^2 \le \int_{\Sigma} \frac{\kappa\left|X^T\right|^2}{u^{m+1}} d s$ & \ref{thm um rn} (ii) ($0 \le m<1$)\\
\hline
$\displaystyle \frac{L^2-2 \pi A}{\pi}+2{\delta_{2, 1}}^2 =\int_{\Sigma} \kappa|X|^2 d s$ & \ref{thm1} \\
\hline
$\displaystyle \int_{\Sigma} \kappa|X|^2 d s - \frac{1}{\pi}\left(L^2-2 \pi A\right) \le \frac{1}{3\pi}\left(L^2-4\pi A\right)$ & \ref{cor isop} ($0=$ Steiner point) \\
\hline
\end{tabular}
\end{table}

\subsection{Results for $k$-convex hypersurfaces}
\begin{theorem}\label{thm mink deficit}
Let $\Sigma=\partial \Omega$ be a smooth, closed hypersurface in $\mathbb{R}^{n+1}$ with $H_k \ge 0$. Then
\begin{align*}
\frac{I_{k-1}{ }^2- I_k I_{k-2}}{I_k}+\delta_{2, k}\left(\Omega, B_0\left(\frac{I_{k-1}}{I_k}\right)\right)^2=\frac{1}{k\binom {n}{k}} \int_{\Sigma} T_{k-1}\circ A\left(X^T, X^T\right)d\mu.
\end{align*}
In particular, we have an upper bound for the deficit in the Alexandrov-Fenchel inequality:
\begin{equation}\label{ineq mink deficit}
\frac{I_{k-1}^2- I_k I_{k-2}}{I_k}\le\frac{1}{k\binom{n}{k}} \int_{\Sigma} T_{k-1} \circ A\left(X^T, X^T\right) d \mu.
\end{equation}
If $H_k>0$, then the equality holds if and only if $\Sigma$ is a sphere centered at the origin.
\end{theorem}

\begin{proof}
This follows from Proposition \ref{prop rn comp} by letting $\varphi(u)=u$. Note that $H_{k-1}\ge 0$ is not needed because when $\varphi(u)=u$, \eqref{L.H.S. first term} in the proof of Proposition \ref{prop rn comp} is simply the Minkowski formula, which holds regardless of the sign of $H_{k-1}$.

If $H_k>0$ and the equality in \eqref{ineq mink deficit} holds, then $\int_{\Sigma}(u-r)^2 H_k d \mu=0$, where $r=\frac{I_{k-1}}{I_k}$. This implies $(u-r)^2\equiv 0$ on $\Sigma$ and hence $\Sigma$ is a sphere of radius $r$ centered at $0$.
\end{proof}

As a corollary of Theorem \ref{thm mink deficit}, we provide a stability result of the geometric inequality \eqref{s1.kw} proved by the authors \cite[Theorem 1.3]{kwong2023geometric} for a closed, convex hypersurface in $\mathbb{R}^{n+1}$.
\begin{corollary}\label{thm Hn x2}
Let $\Sigma=\partial \Omega$ be a smooth, closed convex hypersurface in $\mathbb{R}^{n+1}$. Then
\begin{equation*}
n\delta_2\left(\Omega, B_0\left(\frac{I_{n-1}}{\omega_n}\right)\right)^2 \le \int_\Sigma H_n|X|^2d\mu-\frac{(n+1) I_{n-1}^2-n\omega_n I_{n-2}}{\omega_n}.
\end{equation*}
The equality holds if and only if $\Sigma$ is a sphere centered at the origin.
\end{corollary}

\begin{proof}
We have $T_{n-1} \circ A = H_n I$ and
\begin{equation*}
\frac{\int_\Sigma H_n u \, d\mu}{\int_\Sigma H_n \, d\mu} = \frac{I_{n-1}}{\omega_n}.
\end{equation*}
Therefore, Theorem \ref{thm mink deficit} gives
\begin{align}\label{ineq XT}
\frac{I_{n-1}^2 - \omega_n I_{n-2}}{\omega_n} + \delta_2\left(\Omega, B_0\left(\frac{I_{n-1}}{\omega_n}\right)\right)^2 = \frac{1}{n} \int_\Sigma H_n |X^T|^2 \, d\mu.
\end{align}

On the other hand, using the Cauchy-Schwarz inequality and the Minkowski formula, we have
\begin{equation}\label{ineq X nu}
\int_\Sigma H_n \langle X, \nu \rangle^2 \, d\mu \ge \frac{1}{\int_\Sigma H_n \, d\mu} \left(\int_\Sigma H_n \langle X, \nu \rangle \, d\mu\right)^2 = \frac{1}{\omega_n} I_{n-1}^2.
\end{equation}
Multiplying \eqref{ineq XT} by $n$ and adding it to \eqref{ineq X nu}, we obtain the result.

If the equality holds, then $H_n$ is a non-zero multiple of $H_n \langle X, \nu \rangle$, and hence $\langle X, \nu \rangle$ is constant. Therefore, $\Sigma$ is a sphere centered at the origin.
\end{proof}

In Section \ref{sec. weight}, we will provide another proof of \eqref{s1.kw} by using a completely different method based on the analysis of the support function on the sphere, which additionally offers an explicit interpretation of the deficit in \eqref{s1.kw} in terms of the $L^2$ distance from the ball $B_0\left(\frac{I_{n-1}}{\omega_n}\right)$. See Theorem \ref{thm1}.

\begin{theorem}\label{thm rn 1/u}
Let $\Sigma = \partial \Omega$ be a closed hypersurface in $\mathbb{R}^{n+1}$. Suppose $0$ is the center of the circumscribed sphere of $\Sigma$, and assume further that $\Sigma$ is star-shaped with respect to $0$. Let $r$ denote the circumradius.

(i). If $H_{k-1} > 0$, then
\begin{equation}\label{ineq deficit 2}
\frac{I_{k-1}^2-I_k I_{k-2}}{I_{k-2}}+\frac{1}{r^3} \delta_{2, k-1}\left(\Omega, B_0\left(\frac{I_{k-2}}{I_{k-1}}\right)\right)^2 \le \frac{1}{k\binom{n}{k}} \int_{\Sigma} \frac{T_{k-1} \circ A\left(X^T, X^T\right)}{u^2}.
\end{equation}
The equality holds if and only if $\Sigma$ is a sphere centered at the origin.

(ii). If $H_k>0$, then the Alexandrov-Fenchel ratio $\frac{I_{k-1}^2}{I_k I_{k-2}}$ satisfies
\begin{align}\label{s4.convex-3}
I_{k-1}\log \left(\frac{I_{k-1}^2}{I_k I_{k-2}}\right)+\frac{1}{2 r} \delta_{2, k}\left(\Omega, B_0\left(\frac{I_{k-1}}{I_k}\right)\right)^2+\frac{1}{2 r^2} \delta_{2, k-1}\left(\Omega, B_0\left(\frac{I_{k-2}}{I_{k-1}}\right)\right)^2 \nonumber\\
\le \frac{1}{k\binom{n}{k}} \int_{\Sigma} \frac{T_{k-1} \circ A\left(X^T, X^T\right)}{u} d \mu.
\end{align}
The equality holds if and only if $\Sigma$ is a sphere centered at the origin.
\end{theorem}

\begin{proof}
\textbf{Proof of (i). } Since $H_{k-1}>0$ on $\Sigma$, we have that $\Sigma$ is $(k-1)$-convex by Proposition 3.2 of \cite{BC97}. So both $I_{k-1}$ and $I_{k-2}$ are positive. Let $\varphi(u)=\frac{1}{u}$ and apply Proposition \ref{prop rn comp} to obtain
\begin{align*}
& -\frac{1}{k\binom{n}{k}} \int_\Sigma\frac{T_{k-1}\circ A\left(X^T, X^T\right)}{u^2}d\mu\\
=& I_{k-1}\left(\frac{I_k}{I_{k-1}}-\frac{I_{k-1}}{I_{k-2}}\right) -\int_\Sigma \int_{\frac{I_{k-2}}{I_{k-1}}}^{u(x)} H_{k-1} \varphi^{\prime \prime}(t)(u(x)-t) d t d \mu\\
\le& I_{k-1}\left(\frac{I_k}{I_{k-1}}-\frac{I_{k-1}}{I_{k-2}}\right)- \frac{1}{2}\min_I \varphi''\int_{\Sigma} H_{k-1}(u-\bar u)^2 d \mu\\
=& \frac{I_k I_{k-2}-I_{k-1}{ }^2}{I_{k-2}}- \frac{1}{2}\left(\min_I \varphi''\right)\delta_{2, k-1}\left(\Omega, B_0\left(\bar u\right)\right)^2,
\end{align*}
where
\begin{equation*}
\bar u=\frac{\int_\Sigma H_{k-1} u d\mu}{\int_\Sigma H_{k-1}d\mu}= \frac{I_{k-2}}{I_{k-1}}.
\end{equation*}
We also have
\begin{equation*}
\min_I \varphi''=\frac{2}{\max_\Sigma u^3}=\frac{2}{r^3}.
\end{equation*}
So, rearranging gives
\begin{align*}
\frac{I_{k-1}^2-I_k I_{k-2}}{I_{k-2}}+\frac{1}{r^3} \delta_{2, k-1}\left(\Omega, B_0\left(\frac{I_{k-2}}{I_{k-1}}\right)\right)^2
\le \frac{1}{k\binom{n}{k}} \int_{\Sigma} \frac{T_{k-1} \circ A\left(X^T, X^T\right)}{u^2}.
\end{align*}

We now analyze the case of equality. First, by Taylor's theorem,
\begin{equation}\label{taylor}
\varphi(u(x)) - \varphi(\bar{u})
\ge \varphi'(\bar{u})(u(x) - \bar{u})
+ \frac{1}{2} \min_{I} \varphi'' (u(x) - \bar{u})^2
\end{equation}
holds for all $x\in\Sigma$. Note that $\varphi''$ is strictly decreasing. If $u \not\equiv \bar{u}$, then $\min_{I}\varphi'' < \varphi''(\bar{u})$, and there exists an open set $U\subset \Sigma$ such that $|u(x) - \bar{u}| < |\bar{u} - u(x_0)|$ for $x\in U$, where $\varphi''(u(x_0)) = \min_{I} \varphi''$. By Taylor's theorem, for $x\in U$, we have
$\varphi(u(x)) - \varphi(\bar{u}) > \varphi'(\bar{u})(u(x) - \bar{u}) + \frac{1}{2} \min_I \varphi'' (u(x) - \bar{u})^2$.
Integrating \eqref{taylor} with respect to the measure $H_{k-1}d\mu$ thus gives a strict inequality. In view of the proof of Proposition \ref{prop rn comp}, the equality in \eqref{ineq deficit 2} cannot hold.

Therefore, the equality in \eqref{ineq deficit 2} holds if and only if $u \equiv \bar{u}$, which implies $\Sigma$ is a sphere centered at the origin.

\textbf{Proof of (ii). } By \cite[Proposition 3.2]{BC97}, the condition $H_k>0$ implies that $\Sigma$ is $k$-convex. Then $I_k>0, I_{k-1}>0$ and $I_{k-2}>0$. Take $\varphi(u) = \log u$ in Proposition \ref{prop rn comp}. Since $\min_I \psi''=\frac{1}{r}$ and $\max_I \varphi''=-\frac{1}{r^2}$, we have
\begin{align*}
& I_{k-1}\left(\log\left(\frac{I_{k-1}}{I_k}\right)-\log\left(\frac{I_{k-2}}{I_{k-1}}\right)\right)+
\frac{1}{2 r} \delta_{2, k}\left(\Omega, B_0\left(\frac{I_{k-1}}{I_k}\right)\right)^2\\
& \quad +\frac{1}{2 r^2} \delta_{2, k-1}\left(\Omega, B_0\left(\frac{I_{k-2}}{I_{k-1}}\right)\right)^2
\le\frac{1}{k\binom{n}{k}} \int_{\Sigma} \frac{T_{k-1} \circ A\left(X^T, X^T\right)}{u} d \mu.
\end{align*}

Rearranging, we obtain the desired result.

As $\varphi''$ is strictly increasing, if the equality case holds, the equality case follows from the same type of argument as in Case (i).
\end{proof}

\subsection{Results for convex hypersurfaces}
\begin{theorem}\label{thm tan rn}
Let $\Sigma=\partial \Omega$ be a closed convex hypersurface in $\mathbb{R}^{n+1}$.
Suppose $0$ is the center of the circumscribed sphere of $\Sigma$. Let $r$ denote the circumradius. Then

(i). For $0<m\le 1$,
\begin{align}\label{s4.convex-1}
& I_k^{m-1}\cdot\frac{ I_{k-1}^2- I_k I_{k-2} }{I_k^m+I_{k-1}^m}
+\frac{2+(2-m) r^m}{2\left(1+r^m\right)^2} \delta_{2, k}\left(\Omega, B_0\left(\frac{I_{k-1}}{I_k}\right)\right)^2\nonumber\\
& \qquad+\frac{m}{2\left(1+r^m\right)^2 r^{1-m}} \delta_{2, k-1}\left(\Omega, B_0\left(\frac{I_{k-2}}{I_{k-1}}\right)\right)^2
\le \frac{1}{k\binom{n}{k}} \int_{\Sigma} \frac{T_{k-1} \circ A\left(X^T, X^T\right)}{1+u^m} d \mu.
\end{align}

(ii). For $1<m\le 2$,
\begin{align}\label{s4.convex-2}
& I_k^{m-1} \cdot \frac{I_{k-1}^2-I_k I_{k-2}}{I_k^m+I_{k-1}^m}
+\frac{2+(2-m) r^m}{2\left(1+r^m\right)^2} \delta_{2, k}\left(\Omega, B_0\left(\frac{I_{k-1}}{I_k}\right)\right)^2\nonumber\\
& \qquad
\le \frac{1}{k\binom{n}{k}} \int_{\Sigma} \frac{T_{k-1} \circ A\left(X^T, X^T\right)}{1+u^m} d \mu.
\end{align}
If $H_k>0$, then the equality holds in \eqref{s4.convex-1} or \eqref{s4.convex-2} if and only if $\Sigma$ is a sphere centered at the origin.
\end{theorem}
\begin{proof}
We claim that $0$ cannot lie strictly outside $\Sigma$, and hence $u \ge 0$. If not, then the point $p$ on $\Sigma$ closest to $0$ is non-zero. Therefore, $\Sigma$ lies within $\{x:\langle x, p\rangle > \delta\}$ for some $\delta>0$, and thus $\Sigma$ strictly lies within one half of the circumscribed ball. If we translate the circumscribed sphere $\Gamma$ in the direction of $p$ by $\delta$ units, the resulting sphere $\Gamma^{\prime}$ contains $\Sigma$ strictly in its interior, which contradicts the assumption that $\Gamma$ is the smallest sphere containing $\Sigma$. (See \cite[Lemma 1.11]{GreenOsher2003} for a comprehensive proof, which holds irrespective of the dimension.)

Let $\varphi$ be the function satisfying $\varphi'(u) = \frac{1}{1+u^m}$ with $\varphi(0) = 0$. Define $\psi(u) = u \varphi(u)$.
We compute
$$\varphi''(u) = -\frac{m u^{m-1}}{(1+u^m)^2}\quad
\text{and}\quad
\psi''(u) = 2 \varphi'(u) + u \varphi''(u) = \frac{2 + (2-m) u^m}{(1+u^m)^2}. $$
Thus,
$$\min_I \psi'' = \frac{2 + (2-m)(\max_\Sigma u)^m}{(1+(\max_\Sigma u)^m)^2} = \frac{2+(2-m)r^m}{(1+r^m)^2}. $$
We analyze the maximum of $\varphi''$:
\begin{enumerate}
\item
For $0 < m \le 1$,
$$\max_I \varphi'' = -\frac{m}{(1+(\max_\Sigma u)^m)^2 (\max_\Sigma u)^{1-m}} = -\frac{m}{(1+r^m)^2 r^{1-m}}. $$
\item
For $1 < m \le 2$,
$\max_I \varphi'' \le 0$.
\end{enumerate}
Applying Proposition \ref{prop rn comp}, we obtain the following inequality for $0 < m \le 1$:
\begin{equation}\label{ineq 0<m le 1}
\begin{split}
& \frac{1}{k\binom{n}{k}} \int_{\Sigma} \frac{T_{k-1} \circ A\left(X^{T}, X^{T}\right)}{1+u^m} d \mu\\
\ge& I_{k-1}\left(\varphi\left(\frac{I_{k-1}}{I_k}\right)-\varphi\left(\frac{I_{k-2}}{I_{k-1}}\right)\right)\\
& \quad+\frac{1}{2}\min_I \psi^{\prime \prime} \int_{\Sigma} H_k\left(u-\frac{I_{k-1}}{I_k}\right)^2 d \mu
-\frac{1}{2}\max_I \varphi^{\prime \prime} \int_{\Sigma} H_{k-1}\left(u-\frac{I_{k-2}}{I_{k-1}}\right)^2 d \mu
\\
=& I_{k-1}\left(\varphi\left(\frac{I_{k-1}}{I_k}\right)-\varphi\left(\frac{I_{k-2}}{I_{k-1}}\right)\right)\\
& \quad +\frac{2+(2-m) r^m}{2\left(1+r^m\right)^2} \delta_{2, k}\left(\Omega, B_0\left(\frac{I_{k-1}}{I_k}\right)\right)^2\\
&\quad +\frac{m}{2\left(1+r^m\right)^2 r^{1-m}}\delta_{2, k-1}\left(\Omega, B_0\left(\frac{I_{k-2}}{I_{k-1}}\right)\right)^2.
\end{split}
\end{equation}

In the case where $1<m\le 2$, we have instead
\begin{equation}\label{ineq 1<m le 2}
\begin{split}
& \frac{1}{k\binom{n}{k}} \int_{\Sigma} \frac{T_{k-1} \circ A\left(X^{T}, X^{T}\right)}{1+u^m} d \mu\\
\ge& I_{k-1}\left(\varphi\left(\frac{I_{k-1}}{I_k}\right)-\varphi\left(\frac{I_{k-2}}{I_{k-1}}\right)\right)\\
& \quad +\frac{2+(2-m) r^m}{2\left(1+r^m\right)^2} \delta_{2, k}\left(\Omega, B_0\left(\frac{I_{k-1}}{I_k}\right)\right)^2.
\end{split}
\end{equation}

Let $u_2=\frac{I_{k-1}}{I_k}$ and $u_1=\frac{I_{k-2}}{I_{k-1}}$. Note that $u_2\ge u_1$ by the Alexandrov-Fenchel inequality \eqref{s1.AF2} for convex hypersurface and so by mean value theorem,
\begin{align*}
\varphi(u_2) -\varphi(u_1)=\frac{1}{1+c^n}\left(u_2-u_1\right)
\ge\frac{1}{1+u_2^n}\left(u_2-u_1\right)
\end{align*}
for some $c\in[u_1, u_2]$. Putting this into the inequalities \eqref{ineq 0<m le 1} and \eqref{ineq 1<m le 2} above, we can get the result.

If $H_k > 0$ and the equality case holds, then, as in the proof of Theorem \ref{thm rn 1/u}, the injectivity of $\psi''(u) = \frac{2+(2-n) u^n}{\left(1+u^n\right)^2}$ implies that $u = \bar{u}$ is a constant. Therefore, $\Sigma$ is a sphere centered at the origin.
\end{proof}

\begin{theorem} \label{thm um rn}
Let $\Sigma$ be a strictly convex hypersurface in $\mathbb{R}^{n+1}$. Suppose that $0$ is the center of the circumscribed sphere of $\Sigma$ and let $r$ denote the circumradius. Then

(i). For $0\le m< 1$,
\begin{align} \label{s4.convex-4}
& \left(\frac{I_k}{I_{k-1}}\right)^{1-m} \frac{I_{k-1}^2-I_k I_{k-2}}{I_k}
+\frac{ m+1 }{2 r^{1-m}} \delta_{2, k}\left(\Omega, B_0\left(\frac{I_{k-1}}{I_k}\right)\right)^2\nonumber\\
& \qquad +\frac{ 1-m }{2 r^{2-m}} \delta_{2, k-1}\left(\Omega, B_0\left(\frac{I_{k-2}}{I_{k-1}}\right)\right)^2
\le \frac{1}{k\binom{n}{k}} \int_{\Sigma} \frac{T_{k-1} \circ A\left(X^T, X^T\right)}{u^{1-m}} d \mu.
\end{align}
The equality holds if and only if $\Sigma$ is a sphere centered at $0$.

(ii). For $0{\le}m\le 1$,
\begin{align}\label{s4.convex-5}
& \left(\frac{I_{k-2}}{I_{k-1}}\right)^{1-m} \frac{I_{k-1}{ }^2-I_{k}I_{k-2}}{I_{k-2}}+\frac{ 1-m }{2 r^{m+1}} \delta_{2, k}\left(\Omega, B_0\left(\frac{I_{k-1}}{I_k}\right)\right)^2\nonumber\\
& \qquad+\frac{ m+1 }{2 r^{m+2}} \delta_{2, k-1}\left(\Omega, B_0\left(\frac{I_{k-2}}{I_{k-1}}\right)\right)^2
\le \frac{1}{k\binom{n}{k}} \int_{\Sigma} \frac{T_{k-1} \circ A\left(X^T, X^T\right)}{u^{m+1}} d \mu.
\end{align}
The equality holds if and only if $\Sigma$ is a sphere centered at $0$.
\end{theorem}

\begin{proof}
The proof of Theorem \ref{thm tan rn} shows that $0$ cannot lie outside $\Sigma$. It cannot lie on $\Sigma$ either, for otherwise $\Sigma$ is strictly supported by a hyperplane though the center of the circumscribed sphere $\Gamma$ and so $\Sigma$ can only touch $\Gamma$ in an open hemisphere. A translation argument as in the proof of Theorem \ref{thm tan rn} then shows that $\Gamma$ is not the smallest sphere enclosing $\Sigma$. Therefore $u>0$ on $\Sigma$.

\textbf{Proof of (i). } Assume first $m>0$. Choose $\varphi(u) = u^m$ in Proposition \ref{prop rn comp}. Note that $\psi(u)=u\varphi(u)=u^{m+1}$ is convex, but $u^m$ is concave on $\{u>0\}$. We have
\begin{equation*}
\min_I \psi^{\prime \prime}=\frac{m(m+1)}{ (\max u)^{1-m}}=\frac{m(m+1)}{ r^{1-m}}\ge0
\end{equation*}
and
\begin{equation*}
\max_I \varphi''=\frac{m(m-1)}{r^{2-m}}\le0.
\end{equation*}
So Proposition \ref{prop rn comp} gives
\begin{align*}
& I_{k-1} \left(\left(\frac{I_{k-1}}{I_k}\right)^m - \left(\frac{I_{k-2}}{I_{k-1}}\right)^m\right)
+\frac{m(m+1)}{2r^{1-m}}
\delta_{2, k}\left(\Omega, B_0\left(\frac{I_{k-1}}{I_k}\right)\right)^2 \\
& \quad +\frac{m(1-m)}{2 r^{2-m}} \delta_{2, k-1}\left(\Omega, B_0\left(\frac{I_{k-2}}{I_{k-1}}\right)\right)^2\\
\le& \frac{m}{k\binom{n}{k}} \int_{\Sigma} \frac{T_{k-1} \circ A\left(X^T, X^T\right)}{u^{1-m}} d \mu.
\end{align*}

Let $u_2=\frac{I_{k-1}}{I_k}$ and $u_1=\frac{I_{k-2}}{I_{k-1}}$. Note that $u_2\ge u_1$ by the Alexandrov-Fenchel inequality \eqref{s1.AF2} for convex hypersurface and so by mean value theorem,
\begin{align*}
u_2^m -u_1^m=\frac{m}{c^{1-m}}\left(u_2-u_1\right)
\ge\frac{m}{\left(\frac{I_{k-1}}{I_k}\right)^{1-m}}\left(u_2-u_1\right)
\end{align*}
for some $c\in[u_1, u_2]$. Putting this into the above inequality and dividing it by $m$, we can get the result.

If $m=0$, we can get the inequality by letting $m\downarrow 0$ in \eqref{s4.convex-4}.

For $0<m{<} 1$, as $\varphi''$ is injective, the equality case follows from the same type of argument as in Theorem \ref{thm rn 1/u}. (If $m=1$, this becomes an identity (Theorem \ref{thm mink deficit}).)

If $m=0$ and the equality holds, then
\begin{align*}
\frac{I_{k-1}^2-I_k I_{k-2}}{I_{k-1}}+\frac{1}{2 r} \delta_{2, k}\left(\Omega, B_0\left(\frac{I_{k-1}}{I_k}\right)\right)^2
+\frac{1}{2 r^2} \delta_{2, k-1}\left(\Omega, B_0\left(\frac{I_{k-2}}{I_{k-1}}\right)\right)^2 \\
=\frac{1}{k\binom{n}{k}} \int_{\Sigma} \frac{T_{k-1} \circ A\left(X^T, X^T\right)}{u} d \mu.
\end{align*}
On the other hand, comparing with \eqref{s4.convex-3} it suffices to show that
\begin{equation*}
I_{k-1} \log \left(\frac{I_{k-1}^2}{I_k I_{k-2}}\right)\ge \frac{I_{k-1}^2-I_k I_{k-2}}{I_{k-1}},
\end{equation*}
as the equality case of \eqref{s4.convex-3} will then imply that $\Sigma$ is a sphere centered at $0$. Let $x=\frac{I_{k-1}^2}{I_k I_{k-2}}-1$. Then $x\ge 0$ by the Alexandrov-Fenchel inequality, and it can be seen that it suffices to show that $$\log \left(1+x\right)\ge \frac{x}{1+x}$$ for $x\ge 0$. This inequality is a straightforward exercise in calculus.

\textbf{Proof of (ii). } Let $\varphi(u)=u^{-m}$. Assume first $m>0$. We can proceed as in the proof of \eqref{s4.convex-4} to first obtain
\begin{align*}
& \frac{m}{k\binom{n}{k}} \int_{\Sigma} \frac{T_{k-1} \circ A\left(X^T, X^T\right)}{u^{m+1}} d \mu\\
\ge& m\left(\frac{I_{k-2}}{I_{k-1}}\right)^{1-m} \frac{I_{k-1}{ }^2-I_{k}I_{k-2}}{I_{k-2}}+\frac{m(1-m)}{2 r^{m+1}} \delta_{2, k}\left(\Omega, B_0\left(\frac{I_{k-1}}{I_k}\right)\right)^2\\
& \quad +\frac{m(m+1)}{2 r^{m+2}} \delta_{2, k-1}\left(\Omega, B_0\left({I_{k-2}}{I_{k-1}}\right)\right)^2.
\end{align*}
Dividing this inequality by $m$, we can get the result.
If we let $m \downarrow 0$, we can obtain the result for $m=0$.

If $0<m\le 1$, as $\varphi''$ is injective, the equality case follows from the same type of argument as in Theorem \ref{thm rn 1/u}.
The equality case of \eqref{s4.convex-4} when $m=0$ also implies the same conclusion holds when $m=0$.
\end{proof}

\section{Stability of a weighted geometric inequality}\label{sec. weight}
In a recent paper \cite{kwong2023geometric}, it was proved among other things that for a closed convex hypersurface $\Sigma=\partial \Omega$ in $\mathbb R^{n+1}$, the following inequality holds:
\begin{align}\label{ineq1}
\frac{n+1}{\omega_n}I_{n-1}^2 - n I_{n-2}\le \int_\Sigma H_{n} |X|^2 d\mu.
\end{align}
Equality holds if and only if $\Sigma$ is a coordinate sphere. This is the $k=n$ case of \cite[Theorem 1.3]{kwong2023geometric}.

This inequality can be derived using multiple approaches. We have already seen in Corollary \ref{thm Hn x2} that it follows from the weighted Minkowski formula, which gives an improved version of the result. On the other hand, both \eqref{ineq1} and its two-dimensional counterpart,
\begin{equation}\label{ineq2}
\frac{1}{\pi}\left({L}^2-2 \pi {A}\right) \le \int_{\Sigma} \kappa |X|^2 d s,
\end{equation}
which was first proved in \cite{KWWW2021}, can also be derived using the inverse curvature flow \cite{kwong2023geometric}.

A natural question arising from these results is the stability of \eqref{ineq1}. Specifically, if the deficit
\begin{align*}
\int_{\Sigma} H_{n}|X|^2 d \mu - \frac{n+1}{\omega_n}I_{n-1}^2 + n I_{n-2}
\end{align*}
is small, to what extent does the hypersurface $\Sigma$ approximate a coordinate sphere in a certain sense of closeness? Furthermore, it is interesting to determine the optimal choice of the origin and radius $r$ that would result in the sphere centered at this point, with a radius of $r$, giving the best approximation of $\Sigma$ in the sense that the deficit in \eqref{ineq1} is minimized.

To address these questions, we introduce a third approach based on the analysis of the support function on the sphere, which, in fact, provides the most precise result. Rather than viewing \eqref{ineq1} purely as an inequality, we will derive an exact formula for its deficit. In particular, we will show that the deficit of this inequality corresponds to the square of the $L^2$ distance between $\Omega$ and a ball whose diameter equals the mean width of $\Omega$. Here, the mean width of $\Omega$ is defined as the integral average, across all directions, of the width of $\Omega$ \cite[(2.4.15)]{Groemer1996}.

\subsection{Stability of the inequality \eqref{ineq1}}
\begin{theorem}\label{thm1}
For a closed convex hypersurface $\Sigma=\partial \Omega$ in $\mathbb R^{n+1}$,
\begin{equation*}
\int_{\Sigma} H_{n} |X|^2 d \mu - \frac{n+1}{\omega_n}I_{n-1}^2 + n I_{n-2}
=(n+1)\delta_2\left(\Omega, B_0(r)\right)^2.
\end{equation*}
Here $B_0(r)$ is the ball centered at 0 with the radius $r=\frac{I_{n-1}}{\omega_n}$, which also equals to the half of the mean width $\overline{w}(\Omega)$ of $\Omega$, and $\delta_2$ is the $L^2$ distance between two convex bodies.
\end{theorem}

\begin{proof}
Let $\nabla$, $\Delta$, and $d\theta$ be the gradient, the Laplacian, and the measure with respect to the standard metric on the unit sphere $\mathbb S^{n}\subset \mathbb R^{n+1}$. Let $u$ be the support function of $\Sigma$.
It is known that $\Sigma$ can be parametrized by $u(\xi)\xi+\nabla u(\xi)$, $\xi\in \mathbb S^{n}$, and so $|X|^2=u^2+|\nabla u|^2$. We also have the following well-known formulas:
\begin{equation*}
d \theta = H_{n} d\mu, \quad I_{n-1}=\int_{\Sigma} H_{n-1} d \mu = \int_{\mathbb S^{n}} u d \theta
\end{equation*}
and
\begin{align*}
& I_{n-2}=\int_{\Sigma} H_{n-2} d \mu = \frac{1}{n} \int_{\mathbb S^{n}} u(\Delta u + n u) d \theta
\\
& \qquad=\frac{1}{n} \int_{\mathbb S^{n}} (n u^2 - |\nabla u|^2) d \theta.
\end{align*}
Therefore
\begin{equation*}
\begin{aligned}
& \int_{\Sigma} H_{n} |X|^2 d\mu - \frac{n+1}{\omega_n}I_{n-1}^2 + nI_{n-2} \\
=& \int_{\mathbb S^{n}}\left(u^2 + |\nabla u|^2\right) d\theta - \frac{n+1}{\omega_n} \left(\int_{\mathbb S^{n}} u d\theta \right)^2 + \int_{\mathbb S^{n}} \left(n u^2 - |\nabla u|^2 \right) d\theta \\
=& (n+1) \left(\int_{\mathbb S^{n}} u^2 d \theta - \frac{1}{\omega_n} \left(\int_{\mathbb S^{n}} u d \theta \right)^2 \right) \\
=& (n+1) \int_{\mathbb S^{n}} (u - \overline{u})^2 d \theta
\end{aligned}
\end{equation*}
where $\overline{u}=\frac{1}{|\mathbb S^{n}|}\int_{\mathbb S^{n}}u d\theta$. On the other hand \cite[Sec. 2.4]{Groemer1996},
\begin{equation*}
\overline{w}(\Omega)=\frac{2}{|\mathbb S^{n}|}\int_{\mathbb S^{n}} u d \theta=\frac{2}{\omega_n}\int_\Sigma H_n u d\mu=\frac{2I_{n-1}}{\omega_n}.
\end{equation*}
Therefore
\begin{equation*}
\int_{\mathbb S^{n}}(u-\overline{u})^2 d \theta = \int_{\mathbb S^{n}}(u - u_{B_0(r)})^2 d \theta = \delta_2(\Omega, B_0(r))^2,
\end{equation*}
where $B_0(r)$ is the ball centered at $0$ with the radius $r$ equals to the half of the mean width $\overline{w}(\Omega)$ of $\Omega$ and $\delta_2$ is the $L^2$ distance between two convex bodies.
\end{proof}

Theorem \ref{thm1} serves a dual purpose: it strengthens inequality \eqref{ineq1}, offering a new proof of a result in \cite{kwong2023geometric}, while also characterizing its deficit: it measures the $L^2$ distance between the convex set $\Omega$ and a ball that shares the same mean width.

In the inequality \eqref{ineq1}, clearly the only quantity which depends on the choice of the origin is the weighted curvature integral $\int_{\Sigma} H_{n} |X|^2 d \mu$. It is therefore interesting to look for a choice of the origin which minimizes this integral so as to minimize the deficit. Theorem \ref{thm1} along with the minimization property of the ``Steiner ball'' provides a simple characterization of the optimal origin choice: it turns out the choice is the Steiner point of $\Omega$.

Recall that the Steiner point is defined by
$$
z(\Omega)=\frac{1}{|\mathbb B^{n+1}|} \int_{\mathbb S^{n}} u_\Omega(\xi) \xi\, d \theta\in \mathbb R^{n+1}.
$$
The ball whose center is the Steiner point and whose diameter is the mean width of $\Omega$ will be called the Steiner ball of $\Omega$ and denoted by $B_z(\Omega)$.
\begin{proposition}[{\cite[Proposition 5.1.2]{Groemer1996}}]
If $U$ is an arbitrary ball in $\mathbb R^{n+1}$ and $B_z(\Omega)$ is the Steiner ball of $\Omega$, then
$$
\delta_2(\Omega, U)^2 = \delta_2\left(\Omega, B_z(\Omega)\right)^2 + \delta_2\left(B_z(\Omega), U\right)^2.
$$
Hence, $B_z(\Omega)$ is the unique ball that minimizes $\delta_2(\Omega, U)$, considered as a function of $U$.
\end{proposition}
\begin{corollary}\label{cor steiner pt}
For a closed convex hypersurface $\Sigma=\partial \Omega$ in $\mathbb{R}^{n+1}$, the weighted curvature integral
$\int_{\Sigma} H_{n}|X-X_0|^2 d \mu(X)$ regarded as a function in $X_0\in \mathbb R^{n+1}$ is minimized at the Steiner point $z(\Omega)$. In this case,
$$\int_{\Sigma} H_{n}\left|X-z(\Omega)\right|^2 d \mu = \frac{n+1}{\omega_n}I_{n-1}^2 - n I_{n-2} + (n+1) \delta_2\left(\Omega, B_z(\Omega)\right)^2. $$
\end{corollary}

\medskip
Let $K$ and $L$ be two convex bodies in $\mathbb R^{n+1}$ and let $D$ denote the diameter of $K \cup L$. It is known that the Hausdorff distance $\delta$ and the $L^2$ distance are related by the following inequality (see \cite[Proposition 2.3.1]{Groemer1996}):
$$
c_n D^{-n} \delta(K, L)^{n+2} \le \delta_2(K, L)^2,
$$
where $c_n=\frac{2 |\mathbb B^{n}|}{ (n+1)(n+2)}$.

In view of this result, we can regard the deficit in \eqref{ineq1} as a control of the Hausdorff distance as well.
\begin{corollary}
For any $\varepsilon > 0$ and $D > 0$, there exists $\eta = \eta(\varepsilon, D, n)$ such that if $\Omega$ is a smooth convex body in $\mathbb{R}^{n+1}$ satisfying the following conditions:
\begin{enumerate}
\item
the diameter of $\Omega$ is at most $D$, and
\item
the deficit $\int_{\Sigma} H_{n} |X-z(\Omega)|^2 d \mu - \frac{n+1}{\omega_n}I_{n-1}^2 + nI_{n-2} < \eta$,
\end{enumerate}
then the Hausdorff distance between $\Omega$ and $B_{z}(\Omega)$ is at most $\varepsilon$.
\end{corollary}

\subsection{Comparison with the isoperimetric inequality}
In this section, we show that within the class of convex bodies, the inequality \eqref{ineq1} is stronger than the classical isoperimetric inequality. Theorem \ref{thm1} gives a connection between the deficit in \eqref{ineq1} and the $L^2$ distance between $\Omega$ and its Steiner ball. Combining this with the known result that there is a lower bound for the isoperimetric deficit involving the $L^2$ distance $\delta_2(\Omega, B_z(\Omega))$, we are able to compare the inequality \eqref{ineq1} with the isoperimetric inequality.

For a smooth convex body $\Omega \subset \mathbb R^{n+1}$, we let $\Phi(\Omega)$ denote the isoperimetric deficit of $\Omega$, defined by
$$
\Phi(\Omega):=\left(\frac{ |\partial \Omega| }{|\mathbb S^{n}|}\right)^{n+1} - \left(\frac{ |\Omega| }{|\mathbb B^{n+1}|}\right)^n.
$$
The following result is a strengthened version of the isoperimetric inequality that can be interpreted as stability result in terms of the $L^2$ metric.

\begin{theorem}[{\cite[Theorem 5.3.1]{Groemer1996}}]
If $\Omega$ is a smooth convex body in $\mathbb R^{n+1}$, then
$$
\Phi(\Omega) \ge \eta_n \frac{ |\Omega|^{n} }{ I_{n-2}(\Omega) } \delta_2\left(\Omega, B_z(\Omega)\right)^2
$$
where $\eta_n=\frac{n+2}{n|\mathbb B^{n+1}|^{n}}$.
\end{theorem}

\begin{corollary}\label{cor isop}
Let $\Sigma=\partial \Omega$ be a closed smooth convex hypersurface in $\mathbb{R}^{n+1}$. Assume (by translation) that the Steiner point of $\Omega$ is $0$, then
\begin{align*}
& \int_{\Sigma} H_{n}|X|^2 d \mu - \frac{n+1}{\omega_n}I_{n-1}^2 + n I_{n-2} \\
\le&
\frac{(n+1)I_{n-2}(\Omega)}{\eta_n |\Omega|^{n}} \left[\left(\frac{|\Sigma|}{\left|\mathbb S^{n}\right|} \right)^{n+1} - \left(\frac{ |\Omega| }{ \left| \mathbb B^{n+1} \right| } \right)^n \right].
\end{align*}
\end{corollary}

\section{Generalizations}\label{sec others}

In this section, we discuss some generalizations when the ambient manifold is a space form or a warped product manifold.

\subsection{The space form case}
Let $\Sigma=\partial \Omega$ be a closed, smooth hypersurface in space form $\bar M^{n+1}(K)$, where $K=\pm 1$. We define a weighted curvature integral
\begin{equation*}
    I_{k, m}=\int_{\Sigma} \alpha^m H_k d\mu,
\end{equation*}
where $\alpha=\cosh r$ for $K=-1$ and $\alpha=\cos r$ for $K=1$. We also define $\widetilde{J}_{k}=\int_{\Sigma} u \alpha H_{k+1} d \mu$.
\begin{theorem}\label{s6.thm1}
Let $\Sigma=\partial \Omega$ be a smooth, closed hypersurface in $\bar M^{n+1}(K)$ with $H_k >0$. (If $K=1$, assume $\bar M^{n+1}(1)$ is the open hemisphere.) Let $\varphi$ be a $C^2$ function on $\mathrm{Range}(u)$ and define $\psi(t)=t \varphi(t)$.

(i) If $K=1$ and $\varphi$ is non-decreasing, then
\begin{equation*}
\begin{aligned}
& I_{k-1,1}\left(\varphi\left(\frac{I_{k-1,1}}{I_k}\right)-\varphi\left(\frac{I_{k-2,2}}{I_{k-1,1}}\right)\right) \\
& +\int_{\Sigma}\left(\int_{\frac{I_{k-1,1}}{I_k}}^{u(x)} H_k(x) \psi^{\prime \prime}(t)-\int_{\frac{\widetilde{J}_{k-2}}{I_{k-1,1}}}^{u(x)} \alpha(x) H_{k-1}(x) \varphi^{\prime \prime}(t)\right)(u(x)-t) d t d \mu(x)\\
&\le\frac{1}{k\binom{n}{k}} \int_{\Sigma} \varphi^{\prime}(u) T_{k-1} \circ A\left(X^T, X^T\right)d\mu.
\end{aligned}
\end{equation*}
(ii) If $K=-1$ and $\varphi$ is subadditive, then
\begin{align*}
& I_{k-1,1}\left(\varphi\left(\frac{I_{k-1,1}}{I_k}\right)-\varphi\left(\frac{I_{k-2,2}}{I_{k-1,1}}\right)\right) \\
& +\int_{\Sigma}\left(\int_{\frac{I_{k-1,1}}{I_k}}^{u(x)} H_k(x) \psi^{\prime \prime}(t)-\int_{\frac{\widetilde{J}_{k-2}}{I_{k-1,1}}}^{u(x)} \alpha(x) H_{k-1}(x) \varphi^{\prime \prime}(t)\right)(u(x)-t) d t d \mu(x) \\
&\le\frac{1}{k\binom{n}{k}} \int_{\Sigma} \varphi^{\prime}(u) T_{k-1} \circ A\left(X^T, X^T\right)d\mu \\
&\qquad +I_{k-1,1}\varphi\left(\frac{1}{(k-2)\binom{n}{k-2} I_{k-1,1}} \int_{\Sigma} T_{k-2}\left(X^T, X^T\right) d \mu\right).
\end{align*}
\end{theorem}

\begin{proof}
Applying the weighted Minkowski formula \eqref{weighted mink space form} for $f=\varphi(u)$,  we have
\begin{align}\label{s6.pf1}
\int_{\Sigma} \alpha H_{k-1} \varphi(u) d \mu & +\frac{1}{k\binom{n}{k}} \int_{\Sigma} \varphi^{\prime}(u) T_{k-1} \circ A\left(X^T, X^T\right) \nonumber\\
& =\int_{\Sigma} H_k \varphi(u) u d \mu=\int_{\Sigma} \psi(u) H_k d \mu.
\end{align}
On the one hand, by Lemma \ref{lem jensen}, the first term of \eqref{s6.pf1} satisfies
\begin{equation*}
\begin{split}
& \int_{\Sigma} \alpha H_{k-1} \varphi(u) d \mu\\
=& \varphi\left(\frac{\int_\Sigma u \alpha H_{k-1} d \mu}{\int_\Sigma\alpha H_{k-1} d \mu}\right)\int_\Sigma\alpha H_{k-1} d\mu \\
& +\int_{\Sigma} \alpha(x) H_{k-1}(x) \int_{\frac{\widetilde {J}_{k-2}}{I_{k-1, 1}}}^{u(x)} \varphi^{\prime \prime}(t)(u(x)-t) d t d \mu(x)\\
=& \varphi\left(\frac{\widetilde{J}_{k-2}}{I_{k-1,1}}\right)I_{k-1,1}
 +\int_{\Sigma} \alpha(x) H_{k-1}(x) \int_{\frac{\widetilde {J}_{k-2}}{I_{k-1, 1}}}^{u(x)} \varphi^{\prime \prime}(t)(u(x)-t) d t d \mu(x).
\end{split}
\end{equation*}
Here
\begin{equation}\label{tilde J}
\begin{split}
\widetilde {J}_{k-2}
=& \int_\Sigma u \alpha H_{k-1} d \mu\\
=& \int_\Sigma \alpha^2 H_{k-2}d\mu +\frac{1}{(k-2)\binom{n}{k-2}} \int_\Sigma T_{k-2}\left(\nabla \alpha, X^{T}\right)d\mu\\
=& I_{k-2,2} -\frac{{K}}{(k-2)\binom{n}{k-2}} \int_\Sigma T_{k-2}\left(X^{T}, X^{T}\right)d\mu.
\end{split}
\end{equation}
where in the second equality we used the formula \eqref{weighted mink space form} again.

On the other hand, by Theorem \ref{thm weighted supp}, the right hand side of \eqref{s6.pf1} satisfies
\begin{align*}
\int_\Sigma \psi(u) H_k d\mu
= I_{k-1,1} \varphi\left(\frac{I_{k-1,1}}{I_k}\right)+\int _\Sigma H_k(x) \int_{\frac{I_{k-1,1}}{I_k}}^{u(x)} \psi^{\prime \prime}(t)(u(x)-t) d t d \mu(x).
\end{align*}
Therefore
\begin{equation}\label{eq diff phi}
\begin{split}
&I_{k-1,1} \left(\varphi\left(\frac{I_{k-1,1}}{I_k}\right)
-\varphi\left(\frac{\widetilde{J}_{k-2}}{I_{k-1,1}}\right) \right)\\
&+\int_{\Sigma}\left(\int_{\frac{I_{k-1,1}}{I_k}}^{u(x)} H_k(x) \psi^{\prime \prime}(t)-\int_{\frac{\widetilde{J}_{k-2}}{I_{k-1,1}}}^{u(x)} \alpha(x) H_{k-1}(x) \varphi^{\prime \prime}(t)\right)(u(x)-t) d t d \mu(x)\\
=&\frac{1}{k\binom{n}{k}} \int_{\Sigma} \varphi^{\prime}(u) T_{k-1} \circ A\left(X^T, X^T\right)d\mu.
\end{split}
\end{equation}
We want to compare
$\varphi\left(\frac{I_{k-1,1}}{I_k}\right)$ with $\varphi\left(\frac{I_{k-2,2}}{I_{k-1,1}}\right)$ instead of $\varphi\left(\frac{\widetilde{J}_{k-2}}{I_{k-1,1}}\right)$.

In the case where $K=1$, by monotonicity of $\varphi$ and \eqref{tilde J}, we have
\begin{align}\label{ineq K=1}
\varphi\left(\frac{\widetilde{J}_{k-2}}{I_{k-1,1}}\right)
\le \varphi\left(\frac{ I_{k-2,2}}{I_{k-1,1}}\right).
\end{align}

In the case where $K=-1$, by sub-additivity of $\varphi$ and \eqref{tilde J}, we have
\begin{equation}\label{ineq K=-1}
\begin{split}
\varphi\left(\frac{\widetilde{J}_{k-2}}{I_{k-1,1}}\right)
\le& \varphi\left(\frac{ I_{k-2,2}}{I_{k-1,1}}\right)+ \varphi\left(\frac{1}{(k-2)\binom{n}{k-2} I_{k-1,1}} \int_\Sigma T_{k-2}\left(X^T, X^T\right) d \mu\right).
\end{split}
\end{equation}
Putting \eqref{ineq K=1} or \eqref{ineq K=-1} into \eqref{eq diff phi}, we get the result.
\end{proof}
Although the inequality in Theorem \ref{s6.thm1} may seem complex, substituting different choices of $\varphi$ produces results similar to those in Section \ref{sec upper bounds}. These results still follow the general framework:
\begin{equation*}
\boxed{ \quad \text{deficit} + \text{distance}^2 \le \text{integral}. \quad}
\end{equation*}
For brevity, we omit them here.
\begin{remark}
Motivated by Theorem \ref{s6.thm1}, a natural question arises: does the following Alexandrov-Fenchel type inequality hold for sufficiently convex hypersurfaces in a space form?
\begin{align*}
\left(\int_\Sigma \alpha H_k d\mu\right)^2
\ge \int_\Sigma \alpha^2 H_{k-1}d\mu \int_\Sigma H_{k+1}d\mu.
\end{align*}
\end{remark}

\subsection{The warped product case}

We consider the warped product $(\bar{M}^{n+1}, \bar{g})$, as defined in Section \ref{sec prelim}, which satisfies the following assumption:
\begin{assumA}$\ $
\begin{enumerate}
\item $(N^n, g_N)$ has constant curvature $C$.
\item\label{h2} $\lambda'>0$ in the interior of $I$.
\item\label{h4} $\frac{\lambda^{\prime \prime}}{\lambda}+\frac{C-{\lambda^{\prime}}^2}{\lambda^2}>0$ in the interior of $I$.
\end{enumerate}
\end{assumA}
\noindent Conditions \eqref{h2} and \eqref{h4} correspond to the conditions (H2) and (H4), respectively, in \cite{Brendle2013}. These assumptions are natural, as Brendle \cite{Brendle2013} showed that the Alexandrov theorem holds for $(\bar{M}^{n+1}, \bar{g})$ under Assumption A and some other conditions. Besides space forms, examples of manifolds satisfying these conditions include the de Sitter-Schwarzschild and Reissner-Nordström manifolds.

\begin{theorem}
Suppose $\left(\bar{M}^{n+1}, \bar{g}\right)$ satisfies Assumption A.
Let $\Sigma=\partial \Omega$ be a smooth, closed star-shaped hypersurface in $\bar{M}^{n+1}$ with $H_k \ge 0$ and $H_{k-1}>0$ and let $\varphi$ be a $C^2$ non-decreasing function on $I=\mathrm{Range}(u)$. Then we have
\begin{equation*}
\int_{\Sigma} H_k \varphi(u) d \mu
\ge I_k \varphi\left(\frac{J_{k-1}}{I_k}\right)+\frac{1}{2} \min_I \varphi^{\prime \prime} \delta_{2, k}\left(\Omega, B_0(r)\right)^2,
\end{equation*}
where $r=\frac{\int_{\Sigma} H_k u d \mu}{I_k}$. If $H_k>0$ and $\varphi''$ is injective, then the equality holds if and only if $\Sigma$ is a slice $\{r=\mathrm{constant}\}$.

In particular, if $\varphi$ is convex, then
$$
\int_{\Sigma} H_k \varphi(u) d \mu \ge I_k \varphi\left(\frac{J_{k-1}}{I_k}\right).
$$
\end{theorem}
\begin{proof}
By \cite{kwong2018weighted} Proposition 1 (2) and Lemma 1 (2a), we have that $\mathrm{div}_{\Sigma} T_{k-1} (X^T)\ge0$ and hence \eqref{weighted mink} gives
\begin{align*}
\int_{\Sigma} \alpha H_{k-1}d\mu= \int_\Sigma H_k\langle X, \nu\rangle d\mu-\frac{1}{k\binom{n}{k}} \int_{\Sigma} \mathrm{div}_{\Sigma} T_{k-1} (X^T)d\mu
\le \int_\Sigma H_k\langle X, \nu\rangle d\mu.
\end{align*}
Note that $I_k>0$ by \cite[Proposition 3.2]{BC97}. So Lemma \ref{lem jensen} applied to the probability measure $\frac{H_k d \mu}{I_k}$ gives
\begin{align*}
\frac{\int_{\Sigma} H_k \varphi(u) d \mu}{I_k} \ge& \varphi\left(\frac{\int_{\Sigma} H_k u d \mu}{I_k}\right)+ \frac{1}{2 I_k} \min_I \varphi^{\prime \prime} \int_{\Sigma} H_k(u-\bar u)^2d\mu\\
\ge& \varphi\left(\frac{J_{k-1}}{I_k}\right)+ \frac{1}{2 I_k} \min_I \varphi^{\prime \prime} \int_{\Sigma} H_k(u-\bar u)^2d\mu
\end{align*}
where $\bar u=\frac{\int_{\Sigma} H_k u d \mu}{I_k}$.

If $H_k > 0$ and $\varphi''$ is injective, then equality holds if and only if $\Sigma$ is a slice, following the same argument as in Theorem \ref{thm rn 1/u}.
\end{proof}

\end{document}